\documentclass[11pt,reqno]{amsart}

\newtheorem{proposition}{Proposition}[section]
\newtheorem{lemma}[proposition]{Lemma}

\newtheorem{theorem}[proposition]{Theorem}

\theoremstyle{definition}

\newcommand{\thlabel}[1]{\label{th:#1}}
\newcommand{\thref}[1]{Theorem~\ref{th:#1}}
\newcommand{\selabel}[1]{\label{se:#1}}

\newcommand{\prlabel}[1]{\label{pr:#1}}
\newcommand{\prref}[1]{Proposition~\ref{pr:#1}}

\newcommand{\eqlabel}[1]{\label{eq:#1}}
\newcommand{\equref}[1]{(\ref{eq:#1})}

\newcommand{\Hom}{{\rm Hom}}
\newcommand{\End}{{\rm End}}

\def\ot{\otimes}

\newcommand{\Cc}{\mathcal{C}}

\def\*C{{}^*\hspace*{-1pt}{\Cc}}
\def\text#1{{\rm {\rm #1}}}

\usepackage{amssymb}
\usepackage{color,amssymb,graphicx,amscd}

\begin{document}

\title[Serre Theorem for involutory Hopf algebras] {Serre Theorem for involutory Hopf algebras}

\author{G. Militaru}\thanks{This work was supported by
CNCSIS grant 24/28.09.07 of PN II "Groups, quantum groups, corings
and representation theory". }
\address{Faculty of Mathematics and Computer Science, University of Bucharest, Str.
Academiei 14, RO-010014 Bucharest 1, Romania}
\email{gigel.militaru@fmi.unibuc.ro and gigel.militaru@gmail.com}


\begin{abstract}
We call a monoidal category ${\mathcal C}$ a Serre category if for
any $C$, $D \in {\mathcal C}$ such that $C\ot D$ is semisimple,
$C$ and $D$ are semisimple objects in ${\mathcal C}$. Let $H$ be
an involutory Hopf algebra, $M$, $N$ two $H$-(co)modules such that
$M \otimes N$ is (co)semisimple as a $H$-(co)module. If $N$ (resp.
$M$) is a finitely generated projective $k$-module with invertible
Hattory-Stallings rank in $k$ then $M$ (resp. $N$) is
(co)semisimple as a $H$-(co)module. In particular, the full
subcategory of all finite dimensional modules, comodules or
Yetter-Drinfel'd modules over $H$ the dimension of which is
invertible in $k$ are Serre categories.
\end{abstract}
\date{}
\maketitle

\section*{Introduction}
Around 1950 C. Chevalley \cite[pg. 18]{ch} proved a remarkable
result: if $k$ is a field of characteristic zero, $G$ an abitrary
group and $M$, $N$ two finite dimensional semisimple
$k[G]$-modules then $M\ot N$ is a semisimple $k[G]$-module.
Although the result is a purely algebraic one, the only proof
known up to date is the one given by Chevalley in which special
techniques from algebraic geometry and Lie algebras are used (cf.
\cite[pg. 16]{serre3}). The generalization of the Chevalley
theorem in positive characteristic was given by J.-P. Serre in
\cite{serre1} and the converse in \cite[Theorem 2.4]{serre2}. In
fact, we can formulate these properties in the more general
setting of monoidal categories:

\textbf{Local Chevalley-Serre Property:} \textit{Let $( {\mathcal
C}, \ot \, , I)$ be a monoidal category and $C$, $D$ two objects
of ${\mathcal C}$. Give a necessary and sufficient condition such
that $C\ot D$ is a semisimple object of ${\mathcal C}$}.

Starting from this problem we can derive the following concepts: a
monoidal category ${\mathcal C}$ is called a \textit{Serre
category} if for any $C$, $D \in {\mathcal C}$ such that $C\ot D$
is semisimple, $C$ and $D$ are semisimple objects in ${\mathcal
C}$. A monoidal category $( {\mathcal C}, \ot \, , I)$ is called a
\textit{Chevalley category} if $C\ot D$ is a semisimple object for
any semisimple objects $C$, $D$ of ${\mathcal C}$. The category of
all finite dimensional representations of a group $G$ over a field
of characteristic zero is the basic example of a Chevalley and
Serre category.

This note is dedicated to Serre categories. \prref{serepr} gives a
general example of a Serre categories: a monoidal category
${\mathcal C}$ such that any object has a strong dual and the
functors $C\ot - : {\mathcal C} \to {\mathcal C}$, $ -\ot C :
{\mathcal C} \to {\mathcal C}$ preserve the monomorphisms for all
$C \in {\mathcal C}$ is a Serre category. Together with a couple
of technical lemmas presented further on, this result leads to the
main \thref{teoremaprin}. In particular, if $k$ is a field and $H$
an involutory Hopf algebra, then the full subcategory of all
finite dimensional modules, comodules or Yetter-Drinfel'd modules
over $H$ the dimension of which is invertible in $k$ are Serre
categories.

We recall that a morphism $f: C\to D$ in a category ${\mathcal C}$
is called a split mono if there exists $g : D \to C$ a morphism in
${\mathcal C}$ such that $g\circ f = {\rm Id}_C$. The composition
of two split monos is a split mono and if $h\circ f$ is a split
mono then $f$ is a split mono. An object $D \in {\mathcal C}$ is
called semisimple if any monomorphism $f: C\to D$ is a split mono.

Throughout this note $k$ will be a commutative ring and all
modules, Hopf algebras, tensor products and homomorphisms are over
$k$. A $k$-module $N$ is finitely generated projective if and only
if there exists a (unique) element $R_N = \sum_{i = 1}^n e_i \ot
e_i^* \in N\ot N^*$, called the canonical element of $N$ (or, by
abuse of language, dual basis of $N$) such that
\begin{equation}\eqlabel{1}
\sum _{i = 1}^n e_i^* (n) \, e_i = n \qquad {\rm and} \qquad \sum
_{i = 1}^n n^* (e_i)\, e_i^* = n^*
\end{equation}
for all $n \in N$ and $n^*\in N^*$. We denote by $r_k (N)  = \sum
_{i = 1}^n e_i^* (e_i)\in k$ the (Hattory-Stallings) rank of $N$;
if $k$ is a field then $r_k (N)$ is exactly ${\rm dim} (N) 1_k$.
For a Hopf algebra $H$ we will extensively use Sweedler's
sigma-notation: $\Delta (h) = h_{(1)}\otimes h_{(2)}\in H\otimes
H$ (summation understood). $H$ is called involutory if $S^2 = {\rm
Id}_H$. $( {}_H{\mathcal M}, \ot, k )$ will be the monoidal
category of left $H$-modules and $H$-linear maps: $k$ is a left
$H$-module via the trivial action $h \cdot a = \varepsilon (h) a$,
for all $h\in H$ and $a\in k$ and if $M$ and $N$ are two left
$H$-modules $M\ot N$ is a left $H$-module via the diagonal map: $
h \cdot (m\ot n) := h_{(1)} m \ot h_{(2)} n$, for all $h\in H$,
$m\in M$ and $n\in N$. The $k$-linear dual $N^* = \Hom (N, k)$ is
a left $H$-module via
\begin{equation}\eqlabel{0}
< h \cdot n^*,  n > := <n^*, S(h) n >
\end{equation}
for all $h\in H$, $n^*\in N^*$ and $n\in N$. $( {\mathcal M}^H,
\ot, k )$ will be the monoidal category of right $H$-comodules and
$H$-colinear maps. For a right $H$-comodule $(M, \rho) \in
{\mathcal M}^H$ we denote the coaction by $\rho (m) = m_{<0>} \ot
m_{<1>}$ (summation understood), for all $m\in M$. If $M$ and $N$
are right $H$-comodule then $M\ot N$ is a right $H$-comodule via
$\rho (m\ot n) = m_{<0>} \ot n_{<0>} \ot m_{<1>}n_{<1>}$, for all
$m \in M$, $n\in N$. We recall that a right $H$-comodule is called
cosemisimple if it is a semisimple object in the category
${\mathcal M}^H$. By $H$-module (resp. $H$-comodule) we mean a
left $H$-module (resp. a right $H$-comodule). ${}_H{\mathcal
YD}^H$ will be the monoidal category of (left-right)
Yetter-Drinfel'd modules: an object of it is at once a left
$H$-module and a right $H$-comodule such that
$$
 h_{(1)} m_{<0>}\otimes h_{(2)}m_{<1>}= (h_{(2)}
m)_{<0>}\otimes ( h_{(2)}  m)_{<1>}h_{(1)}
$$
for all $h\in H$ and $m\in M$. The morphisms of ${}_H{\mathcal
YD}^H$ are $H$-linear and $H$-colinear maps.

\section{Serre Theorem for involutory Hopf algebras}\selabel{2}
Let $( {\mathcal C}, \ot \, , I)$ be a monoidal category. We shall
say that an object $D \in {\mathcal C}$ has a \textit{strong right
dual} if there exists an object $D^* \in {\mathcal C}$ and a split
mono $i_D : I \to D\ot D^*$ in ${\mathcal C}$. An object $C \in
{\mathcal C}$ has a \textit{strong left dual} if there exists an
object $C^* \in {\mathcal C}$ and a split mono $i_C : I \to C^*\ot
C$ in ${\mathcal C}$. A monoidal category $( {\mathcal C}, \ot \,
, I)$ is called a \textit{Serre category} if for any $C$, $D \in
{\mathcal C}$ such that $C\ot D$ is semisimple we have that $C$
and $D$ are semisimple objects in ${\mathcal C}$.

\begin{proposition}\prlabel{serepr}
Let $( {\mathcal C}, \ot \, , I)$ be a monoidal category, $C$ and
$D$ two objects in ${\mathcal C}$ such that $C\ot D$ is a
semisimple object in ${\mathcal C}$. The following holds:
\begin{enumerate}
\item If $D$ has a strong right dual and the functor $ - \ot D :
{\mathcal C} \to {\mathcal C}$ preserves monomorphisms then $C$ is
semisimple in ${\mathcal C}$.

\item If $C$ has a strong left dual and the functor $C\ot - :
{\mathcal C} \to {\mathcal C}$ preserves monomorphisms then $D$ is
semisimple in ${\mathcal C}$.
\end{enumerate}
In particular, if any object  of ${\mathcal C}$ has strong left
and right duals and the functors $C\ot - : {\mathcal C} \to
{\mathcal C}$, $ -\ot C : {\mathcal C} \to {\mathcal C}$ preserve
monomorphisms for all $C \in {\mathcal C}$ then ${\mathcal C}$ is
a Serre category.
\end{proposition}

\begin{proof}
1) Let $ j : C' \to C$ be a monomorphism in ${\mathcal C}$. Then
$j \ot {\rm Id}_D : C'\ot D \to C\ot D$ is a monomorphism in
${\mathcal C}$, hence is a split mono as $C\ot D$ is semisimple.
Consider the commutative diagram
$$
\begin {CD}
C' \cong C'\ot I @> {\rm Id}_{C'}\ot i_D >> C'\ot D \ot D^*\\
@VVj V @VV j\ot {\rm Id}_D \ot {\rm Id}_{D^*} V\\
C\cong C\ot I @> {\rm Id}_C \ot i_D >> C\ot D \ot D^*
\end{CD}
$$
Now, ${\rm Id}_{C'}\ot i_D$ and $j\ot {\rm Id}_D \ot {\rm
Id}_{D^*}$ are split monos; hence $ ({\rm Id}_C \ot i_D)\circ j$
is a split mono, thus $j$ is a split mono i.e. $C$ is semisimple.

2) Let $ j : D' \to D$ be a monomorphism in ${\mathcal C}$. Then
${\rm Id}_C \ot j : C\ot D' \to C\ot D$ is a monomorphism in
${\mathcal C}$, hence it is a split mono as $C\ot D$ is
semisimple. Consider the commutative diagram
$$
\begin {CD}
D' \cong I\ot D' @> i_C \ot {\rm Id}_{D'}  >> C^*\ot C \ot D'\\
@VVj V @VV  {\rm Id}_{C^*} \ot {\rm Id}_C \ot j  V\\
D\cong I\ot D @> i_C \ot {\rm Id}_D >> C^*\ot C \ot D
\end{CD}
$$
where $i_C \ot {\rm Id}_{D'}$ and ${\rm Id}_{C^*} \ot {\rm Id}_C
\ot j$ are split monos; hence $ (i_C \ot {\rm Id}_D) \circ j $ is
a split mono, thus $j$ is a split mono i.e. $D$ is semisimple.
\end{proof}

In order to give the main examples of Serre categories we need
some technical results:

\begin{lemma}
Let $H$ be a Hopf algebra, $N$ be a right $H$-comodule, finitely
generated and projective over $k$. Then $N^* = \Hom (N, k)$ has a
structure of right $H$-comodule via the coaction given by
\begin{equation}\eqlabel{coac}
\rho_{N^*} (e_i ^*) = \sum_{j=1}^n \, e_j ^* \ot S \bigl(
(e_{j})_{<1>} \bigl ) \, e_i ^* \bigl( (e_{j})_{<0>} \bigl)
\end{equation}
for all $i = 1, \cdots, n$, where $\sum_{i = 1}^n e_i \ot e_i^*
\in N\ot N^*$ is the canonical element of $N$.
\end{lemma}

\begin{proof}
This is a well-know structure and the proof is just a
straightforward verification. Using the $\sum$-notation the
coaction given by \equref{coac} can be written as follows:
$$
\rho_{N^*} (f) = f_{<0>} \ot f_{<1>} \in N^* \ot H \quad {\rm iff}
\quad f_{<1>} f_{<0>} (n) = S(n_{<1>}) f (n_{<0>})
$$
for all $n\in N$ and $f \in N^*$.
\end{proof}

Let $N\in {}_H{\mathcal YD}^H$ be a Yetter-Drinfel'd module,
finitely generated and projective over $k$; then $N^* = \Hom (N,
k) \in {}_H{\mathcal YD}^H$ via the $H$-module structure given by
\equref{0} and the $H$-comodule structure given by \equref{coac}:
this is the left-right version of \cite[Proposition 4.4.2]{lara}.

\begin{proposition}\prlabel{serepropn}
Let $H$ be an involutory Hopf algebra. Then:
\begin{enumerate}
\item Any $H$-module (resp. $H$-comodule) that is finitely
generated projective as a $k$-module with invertible rank in $k$
has a right and a left strong dual in the category ${}_H{\mathcal
M}$ (resp. ${\mathcal M}^H$). \item Any Yetter-Drinfel'd module
that is finitely generated projective as a $k$-module with
invertible rank in $k$ has a right and a left strong dual in the
category ${}_H{\mathcal YD}^H$.
\end{enumerate}
\end{proposition}

\begin{proof}
1) Let $N$ be a $H$-module that is finitely generated projective
as a $k$-module with invertible rank $r_k (N)$ in $k$. We shall
prove that $N^*$ with the structure given by \equref{0} is a
strong right and left dual of $N$ in ${}_H{\mathcal M}$. We view
the canonical element $R_N$ as a $k$-linear map
$$
i_N : k \to N\ot N^*, \qquad i_N (a) := a \, R_N  = a \sum_{i =
1}^n e_i \ot e_i^*
$$
for all $a \in k$ and we shall prove that $i_N$ is an $H$-module
map, i.e. we have to prove that
\begin{equation}\eqlabel{2}
\varepsilon (h) \sum_{i = 1}^n e_i \ot e_i^* = \sum_{i = 1}^n
h_{(1)} e_i \ot h_{(2)} \cdot e_i^*
\end{equation}
for all $h\in H$. We shall use the $k$-linear isomorphism
$$
\varphi_{N} : N\ot N^* \to \End (N), \quad \varphi_{N} (n\ot n^*)
(n') := \, <n^*, n'>\, n
$$
given by the fact that $N$ is finitely generated projective over
$k$. For $n\in N$ we have:
$$
\varphi_{N} \bigl( \varepsilon (h) \sum_{i = 1}^n e_i \ot e_i^*
\bigl) = \varepsilon (h) \sum _{i = 1}^n e_i^* (n) \, e_i
\stackrel{ \equref{1} } {=} \varepsilon (h) n
$$
and
\begin{eqnarray*}
\varphi_{N} \bigl(\sum_{i = 1}^n h_{(1)} e_i \ot h_{(2)} \cdot
e_i^* \bigl) (n) &=& \sum_{i = 1}^n  \bigl( h_{(2)} \cdot e_i^*
\bigl) (n) \, h_{(1)} e_i \\
&\stackrel{ \equref{0} } {=}& \sum_{i = 1}^n h_{(1)}\, e_i^*
\bigl( S
(h_{(2)}) n \bigl)\, e_i \\
&\stackrel{ \equref{1} } {=}& h_{(1)} S( h_{(2)} ) n = \varepsilon
(h) n
\end{eqnarray*}
hence \equref{2} holds and $i_N : k \to N\ot N^*$ is $H$-linear.
Let us consider now the evaluation map
$$
ev_N : N\ot N^* \to k, \qquad ev_N (n \ot n^*) = <n^*, n>
$$
for all $n\in N$ and $n^*\in N^*$. As $H$ is a involutory Hopf
algebra, $ev_N$ is a $H$-linear map. Indeed,
\begin{eqnarray*}
ev_N \bigl( h\cdot (n\ot n^*) \bigl)  &=& ev_N ( h_{(1)} n \ot
h_{(2)} \cdot n^* )\\
&=& < h_{(2)} \cdot n^*  , h_{(1)} n > \\
&=& < n^*, S(h_{(2)}) h_{(1)}
n > \\
&\stackrel{ (S = S^{-1}) } {=}& < n^*, \varepsilon (h) n > \\
&=& \varepsilon (h)\, ev_N ( n\ot n^* ) \\
&=& h\cdot ev_N ( n \ot n^* )
\end{eqnarray*}
for all $h\in H$, $n\in N$ and $n^* \in N^*$. Thus we have two
$H$-module maps
$$
i_N : k \to N\ot N^*, \qquad ev_N : N\ot N^* \to k
$$
such that
\begin{equation}\eqlabel{4}
ev_N \circ i_N = r_k (N)
\end{equation}
As $r_k (N)$ is invertible in $k$ the map $i_N : k \to N\ot N^*$
is a split mono in the category ${}_H{\mathcal M}$: the $H$-linear
retraction follows from \equref{4} and is given by $r_k (N) ^{-1}
ev_N:  N\ot N^* \to k$. Thus $N^*$ is a strong right dual of $N$
in ${}_H{\mathcal M}$.

We shall prove now that $N^*$ is also a strong left dual of $N$ in
${}_H{\mathcal M}$. We consider the $k$-linear map $j_N : k \to
N^* \ot N$, $j_N (a) := a \sum_{i=1}^n e_i^* \ot e_i$, for all $a
\in k$. Using the fact that $H$ is an involutory Hopf algebra we
can easily show that $i_N$ is an $H$-module map: here we use the
$k$-linear isomorphism
$$
\psi_{N}: N^*\ot N \to \End (N), \quad \psi_{N} (n^*\ot n) (n') :=
<n^*, n'> n
$$
given by the fact that $N$ is finitely generated and projective
over $k$. As $r_k(N)$ is invertible in $k$ the map $j_N : k \to
N^*\ot N$ is split mono in ${}_H{\mathcal M}$: the $H$-linear
retraction is given by $r_k(M) ^{-1} ev'_N: N^*\ot N \to k$, where
$ev'_N : N^* \ot N \to k, \quad ev'_N (n^* \ot n) := <n^*, n>$,
for all $n^*\in N^*$, $n\in N$ is the evaluation map and it is an
$H$-module map.

We shall consider now the comodule case. Let $N$ be an
$H$-comodule that is finitely generated projective as a $k$-module
with invertible rank $r_k (N)$ in $k$. We shall prove that $N^*$
with the structure given by \equref{coac} is a strong right and
left dual of $N$ in ${\mathcal M}^H$. First we prove that $i_N : k
\to N\ot N^*$, $i_N (a) = a \sum_{i = 1}^n e_i \ot e_i^* $, for
all $a \in k$ is an $H$-comodule map, where $k$ is an $H$-comodule
via the trivial coaction $\rho (a) = a \ot 1_H$. Thus we have to
prove the formula:
\begin{equation}\eqlabel{33}
\sum_{i=1}^n \, (e_{i})_{<0>} \ot (e_{i}^*)_{<0>} \ot
(e_{i})_{<1>} (e_{i}^*)_{<1>} = \sum_{i=1}^n \, e_i \ot e_i^* \ot
1_H
\end{equation}
If we apply $\rho$ in \equref{1} we obtain
\begin{equation}\eqlabel{44}
x_{<0>} \ot x_{<1>} = \sum_{i=1}^n \, e_i^* (x) (e_{i})_{<0>} \ot
(e_{i})_{<1>}
\end{equation}
for all $x\in N$. We denote LHS the left hand side of \equref{33}.
We have:
\begin{eqnarray*}
{\rm LHS} &\stackrel{ \equref{coac}} {=}& \sum_{i, j = 1}^n  \,
(e_{i})_{<0>} \ot e_j^* \ot
(e_{i})_{<1>} S\bigl( (e_{j})_{<1>} \bigl) e_i^* ( (e_{j})_{<0>} )\\
&=& \sum_{i, j = 1}^n  \, e_i^* ( (e_{j})_{<0>} ) (e_{i})_{<0>}
\ot e_j^* \ot (e_{i})_{<1>} S\bigl( (e_{j})_{<1>} \bigl) \\
&\stackrel{ \equref{44} } {=}& \sum_{j=1}^n (e_{j})_{<0>} \ot
e_j^* \ot (e_{j})_{<1>} S\bigl( (e_{j})_{<2>} \bigl)\\
&=&  \sum_{j=1}^n (e_{j})_{<0>} \ot e_j^* \ot \varepsilon \bigl( (e_{j})_{<1>} \bigl)1_H  \\
&=& \sum_{j=1}^n \, e_j \ot e_j^* \ot 1_H
\end{eqnarray*}
hence \equref{33} holds i.e. $ i_N : k \to N\ot N^*$, is an
$H$-comodule map. Using the fact that $H$ is involutory we shall
prove that the map $ev_N : N\ot N^* \to k$, $ev_N (n \ot n^*) =
<n^*, n>$, for all $n\in N$, $n^*\in N^*$ is also an $H$-comodule
map. This is equivalent to showing that:
\begin{equation}\eqlabel{55}
< (e_j^*)_{<0>} , \, (e_{i})_{<0>} > \ot \, (e_{i})_{<1>}
(e_j^*)_{<1>} = \delta_{ij} \ot 1_H
\end{equation}
for all $i$, $j = 1, \cdots, n$, where $\delta_{ij}$ is the
Kronecker sign. We denote LHS the left hand side of \equref{55}.
We have:
\begin{eqnarray*}
{\rm LHS} &\stackrel{ \equref{coac}} {=}& \sum_{t = 1}^n  \,
<e_t^*  ,\, (e_{i})_{<0>} > \ot \, (e_{i})_{<1>} S\bigl(
(e_{t})_{<1>} \bigl) e_j^* ( (e_{t})_{<0>} )\\
&=& \sum_{t = 1}^n  \, 1 \ot (e_{i})_{<1>} S\bigl( (e_{t})_{<1>}
\bigl) e_j^* \bigl( <e_t^*  ,\, (e_{i})_{<0>} > (e_{t})_{<0>}
\bigl)\\
&\stackrel{ \equref{44} } {=}& 1\ot (e_{i})_{<2>} S\bigl(
(e_{i})_{<1>} \bigl) e_j^* ( (e_{i})_{<0>} )\\
&\stackrel{ (S = S^{-1}) } {=}&  1 \ot \varepsilon ((e_{i})_{<1>})
1_H e_j^* ( (e_{i})_{<0>} ) \\
&=& 1 \ot e_j^* (e_i) 1_H = \delta_{ij} \ot 1_H
\end{eqnarray*}
hence \equref{55} holds and the evaluation map $ev_N : N\ot N^*
\to k$ is right $H$-colinear. Thus we have two $H$-comodule maps
$i_N : k \to N\ot N^*$ and $ev_N : N\ot N^* \to k$ such that $ev_N
\circ i_N = r_k (N)$. As $r_k (N)$ is invertible in $k$ the map
$i_N : k \to N\ot N^*$ is a split mono in the category ${\mathcal
M}^H$, i.e. $N^*$ is a strong right dual of $N$ in ${\mathcal
M}^H$.

Finally, we have to prove that $N^*$ is also a strong left dual of
$N$ in ${\mathcal M}^H$. The proof is similar to the computations
above. As $H$ is involutory and $r_k(N)$ is invertible in $k$ the
map $j_N : k \to N^*\ot N$ is split mono in ${\mathcal M}^H$: the
$H$-colinear retraction is given by $r_k(N) ^{-1} ev'_N: N^*\ot N
\to k$, where $ev'_N : N^* \ot N \to k, \quad ev'_N (n^* \ot n) :=
<n^*, n>$, for all $n^*\in N^*$, $n\in N$ is the evaluation map
and it is an $H$-comodule map.

2) If $N \in {}_H{\mathcal YD}^H$ then $N^* \in {}_H{\mathcal
YD}^H$ is a right and left dual of $N$ in the category of
Yetter-Drinfel'd modules: this follows from 1) as the maps $i_N :
k \to N\ot N^*$ and $j_N : k \to N^*\ot N$ are split mono as
$H$-linear and $H$-colinear maps, i.e. they are split mono in the
category ${}_H{\mathcal YD}^H$.
\end{proof}

Using \prref{serepr} and \prref{serepropn} we obtain the main
results of this note:

\begin{theorem}\thlabel{teoremaprin}
Let $H$ be an involutory Hopf algebra, $M$ and $N$ two
$H$-(co)modules such that $M \otimes N$ is (co)semisimple as a
$H$-(co)module. The following hold:
\begin{enumerate}
\item If $N$ is finitely generated projective as a $k$-module with
invertible rank in $k$ then $M$ is (co)semisimple as a
$H$-(co)module.

\item If $M$ is finitely generated projective as a $k$-module with
invertible rank in $k$ then $N$ is (co)semisimple as a
$H$-(co)module.
\end{enumerate}
In particular, if $k$ is a field, the full subcategory of all
finite dimensional $H$-modules (reps. $H$-comodules) that have
invertible dimension in $k$ is a Serre category.
\end{theorem}

\begin{proof}
\prref{serepropn} shows that the hypotheses of \prref{serepr} are
true: we note that for a projective $k$-module $X$ the functors $X
\ot -$, $-\ot X$ preserve monomorphisms.
\end{proof}

Similarly we have obained the following:

\begin{theorem}\thlabel{teoremaprin'}
Let $H$ be an involutory Hopf algebra, $M$ and $N$ two
Yetter-Drinfel'd modules such that $M \otimes N$ is a semisimple
object of ${}_H{\mathcal YD}^H$. The following hold:
\begin{enumerate}
\item If $N$ is finitely generated projective as a $k$-module with
invertible rank in $k$ then $M$ is a semisimple object of
${}_H{\mathcal YD}^H$. \item If $M$ is finitely generated
projective as a $k$-module with invertible rank in $k$ then $N$ is
a semisimple object of ${}_H{\mathcal YD}^H$.
\end{enumerate}
In particular, if $k$ is a field, the full subcategory of all
finite dimensional Yetter-Drinfel'd modules that have invertible
dimension in $k$ is a Serre category.
\end{theorem}

More difficult is the converse of the last statement of
\thref{teoremaprin}, which is a generalization of the Chevalley
theorem for involutory Hopf algebras (we only state this for
characteristic zero):

\textbf{Question:} \textit{Let $k$ be a field of characteristic
zero, $H$ an involutory Hopf algebra and $M$ and $N$ two finite
dimensional (co)semisimple $H$-(co)modules. Is $M\ot N$
(co)semisimple as a $H$-(co)module?}

The answer is affirmative for $H = k[G]$ based on the Chevalley
theorem and the fact that $k[G]$ is cosemisimple as a coalgebra.
In the case that $H$ is finite dimensional: an involutory finite
dimensional Hopf algebra $H$ is semisimple and cosemisimple
following the well known theorem of Larson and Radford. In
generally, the Chevalley theorem fails for arbitrary
finite-dimensional Hopf algebras, an example being the
Frobenius-Lusztig kernels (\cite{aeg}).

\end{document}